\documentclass{amsart}
\usepackage{amssymb,amsmath, tikz, epsfig, float}
\usepackage{amsthm}
\usepackage{mathrsfs}
\usepackage{eucal}
\usepackage{enumitem, indentfirst}
\usepackage[fleqn,tbtags]{mathtools}
\usepackage{color}
\usepackage{cite}

\usepackage{color}
\usepackage{xcolor}
\usepackage{centernot}

\usepackage{mathrsfs}

 \theoremstyle{plain}

 \newtheorem{theorem}{Theorem}

  \numberwithin{equation}{section}

\theoremstyle{remark}

\newcommand{\hil}{\mathcal{H}}
\newcommand{\bh}{\mathbf{B}(\hil)}
\DeclarePairedDelimiterX\sip[2]{\langle}{\rangle}{#1\delimsize,#2}
\newcommand{\bit}{\begin{itemize}}
\newcommand{\eit}{\end{itemize}}
\newcommand{\ben}{\begin{enumerate}}
\newcommand{\een}{\end{enumerate}}
\newcommand{\be}{\begin{equation}}
\newcommand{\ee}{\end{equation}}
\newcommand{\ba}{\begin{array}}
\newcommand{\ea}{\end{array}}

\newcommand\cM{\mathcal M}
\newcommand\cN{\mathcal N}
\newcommand\cK{\mathcal K}

\newcommand{\fel}{1/2}

\newcommand{\rng}{\operatorname{ran}}
\newcommand{\ran}{\operatorname{ran}}

\newcommand{\lers}[1]{\left\{ #1 \right\}}

\newcommand{\nll}{\centernot{\ll}}

\newcommand{\ranfaf}{\rng\varphi(A)^{\fel}}

\newcommand{\Afel}{A^{\fel}}

\newcommand{\ranaf}{\rng A^{\fel}}
\newcommand{\rancf}{\rng C^{\fel}}
\newcommand{\ranbf}{\rng B^{\fel}}

\newcommand{\RM}{\mathcal{R}_{1/2}(\cM)}
\newcommand{\RTM}{\mathcal{R}_{1/2}(T(\cM))}
\newcommand{\ffi}{\varphi}
\renewcommand{\phi}{\varphi}
\newcommand{\fa}{\varphi(A)}
\newcommand{\fb}{\varphi(B)}

\newcommand{\lath}{\operatorname{Lat}(\hil)} 
\newcommand{\latoph}{\operatorname{Lat}_{\mathrm{op}}(\hil)}  
\newcommand{\lateh}{\operatorname{Lat}_{1}(\hil)}
\newcommand{\latkh}{\operatorname{Lat}_{2}(\hil)}
\newcommand{\latnh}{\operatorname{Lat}_{n}(\hil)}
\newcommand{\latcnh}{\operatorname{Lat}_{-n}(\hil)}
\newcommand{\latcegyh}{\operatorname{Lat}_{-1}(\hil)}
\newcommand{\codim}{\operatorname{codim}}

\newcommand{\bph}{\mathbf{B}_+(\hil)}
\newcommand{\seto}{\{0\}}

\newcommand{\qiff}{\quad\iff\quad}

\title[Maps preserving $<<$ and  $\perp$ of positive operators]{Maps preserving absolute continuity and singularity of positive operators}
                                                           
\author[Gy\"orgy P\'al Geh\'er]{Gy\"orgy P\'al Geh\'er}
\address{Gy\"orgy P\'al Geh\'er, Department of Mathematics and Statistics\\ University of Reading\\ Whiteknights\\ P.O.
Box 220\\ Reading RG6 6AX\\ United Kingdom}
\email{G.P.Geher@reading.ac.uk or gehergyuri@gmail.com \newline
\hspace{1.6cm} http://www.math.u-szeged.hu/\~{}gehergy}

\author[Zsigmond Tarcsay]{Zsigmond Tarcsay}
\address{Zsigmond Tarcsay, Department of Applied Analysis  and Computational Mathematics\\ E\"otv\"os Lor\'and University\\ P\'azm\'any P\'eter s\'et\'any 1/c.\\ Budapest H-1117\\ Hungary}

\email{tarcsay@cs.elte.hu\newline
http://tarcsay.web.elte.hu/}

\author[Tam\'as Titkos]{Tam\'as Titkos}
\address{Tam\'as Titkos, Alfr\'ed R\'enyi Institute of Mathematics\\ Hungarian Academy of Sciences\\ Re\'altanoda u. 13-15.\\
Budapest H-1053\\ Hungary\\ and BBS University of Applied Sciences\\ Alkotm\'any u. 9.\\
Budapest H-1054\\ Hungary}

\email{titkos.tamas@renyi.mta.hu \newline http://renyi.hu/\~{}titkos}

\begin{document}



\thanks{Gy. P. Geh\'er was supported by the Leverhulme Trust Early Career Fellowship (ECF-2018-125), and also by the Hungarian National Research, Development and Innovation Office (Grant no. K115383).}
\thanks{Zs. Tarcsay was supported by DAAD-TEMPUS Cooperation Project ``Harmonic Analysis and Extremal Problems'' (grant no. 308015) and by Thematic Excellence Programme, Industry and Digitization Subprogramme, NRDI Office, 2019.}
\thanks{T. Titkos was supported by the Hungarian National Research, Development and Innovation Office  NKFIH (grant no. PD128374 and grant no. K115383), by the J\'anos Bolyai Research Scholarship of the Hungarian Academy of Sciences, and by the \'UNKP-19-4-BGE-1 New National Excellence Program of the Ministry for Innovation and Technology
.}

\begin{abstract}
In this paper we consider the cone of all positive, bounded operators acting on an infinite dimensional, complex Hilbert space, and examine bijective maps that preserve absolute continuity in both directions. 
It turns out that these maps are exactly those that preserve singularity in both directions. 
Moreover, in some weak sense, such maps are always induced by bounded, invertible, linear- or conjugate linear operators of the underlying Hilbert space. 
Our result gives a possible generalization of a recent theorem of Molnar which characterizes maps on the positive cone that preserve the Lebesgue decomposition of operators.
 
\end{abstract}
\maketitle

\section{Introduction} 
Throughout this paper $\hil$ will denote a complex infinite dimensional Hilbert space, unless specifically stated otherwise, with the inner product $(\cdot\,|\,\cdot)$. The symbols $\bh$ and $\bph$ will stand for the set of all bounded operators and the cone of all positive operators, respectively. Motivated by their measure theoretic analogues, Ando introduced the notion of absolute continuity and singularity of positive operators in \cite{andoeq}, and proved a Lebesgue decomposition theorem in the context of $\bph$. Since then similar results have been proved in more general contexts, we only mention a few of them: \cite{HSdS,JuryMartin,KosakiLebesgue,TZS-Leb,pems,Voiculescu}.

\par Given a mathematical structure and an important operation/quantity/relation corresponding to it, a natural question to ask is: how can we describe all maps that respect this operation/quantity/relation? Such and similar problems belong to the gradually enlarging field of \emph{preserver problems}, the interested reader is referred to the survey papers \cite{GLS,LT,LP} for an introduction. A considerable part of preserver problems is related to operator structures, for which we refer to the book of Moln\'ar \cite{MolnarBook} and the reference therein. 

\par In this paper our goal is to generalize Moln\'ar's result \cite[Theorem 1.1]{MolnarLebesgue} about the structure of bijective maps on $\bph$ that preserve the Lebesgue decomposition in both directions. Moln\'ar proved that the cone is quite rigid in the sense that these maps can be always written in the form $$A\mapsto SAS^*$$ with a bounded, invertible, linear- or conjugate linear operator $S\colon \hil\to\hil$. A natural question arises: how can we describe the form of those bijections that preserve absolute continuity (or singularity) of operators in both directions? Clearly, this is a weaker condition than that of Moln\'ar, hence maps considered by Moln\'ar obviously preserve this relation. However, it is not too hard to construct other maps which preserve absolute continuity. For example, one could use the fact that every positive operator is absolutely continuous with respect to every invertible element of $\bph$, and that invertible elements are the only ones with this property. Therefore, if we leave all positive and not invertible operators fixed, and consider an arbitrary bijection on the subset of invertible and positive operators, then this map preserves absolute continuity in both directions. Despite the existence of such seemingly unstructured maps, it is still possible to describe all maps with this weaker preserver property.

\section{Technical preliminaries}

We say that a bounded linear operator $A\colon \hil\to\hil$ is positive if $(Ax\,|\,x)\geq0$ holds for all $x\in \hil$. This notion induces a partial order on $\bph$, that is, $A\leq B$ if $B-A\in\bph$. Two positive operators $A,B\in\bph$ are said to be \emph{singular}, $A\perp B$ in notation, if the only element $C\in\bph$ with $C\leq A$ and $C\leq B$ is the zero operator. It turns out that this relation can be phrased in terms of the ranges of the positive square roots (see \cite[p. 256]{andoeq}):
\begin{equation}\label{d:singrng}
A\perp B\qiff\ranaf\cap\ranbf=\seto.
\end{equation}

Next, $A$ is said to be \emph{$B$-dominated} if there exists a $c\geq0$ such that $A\leq cB$. If $A$ can be approximated by a monotone increasing sequence of $B$-dominated operators in the strong operator topology then, we say that \emph{$A$ is $B$-absolutely continuous}, and we write $A\ll B$. Observe that this definition of absolute continuity combined with the Douglas factorization theorem \cite[Theorem 1]{douglas} yields 
\begin{equation*}
    A\ll B\quad\Longrightarrow \quad \overline{\ran A} \subseteq \overline{\ran B},
\end{equation*}
however, the converse implication is not true in general (see e.g. \cite[Example 3]{TZS-Leb}). A characterization of absolute continuity by means of operator ranges reads as follows (see \cite[Theorem 5]{andoeq}):
\begin{equation}\label{d:acrng}
A\ll B\qiff \{x\in\hil : \Afel x\in\ranbf\}\;\mbox{is dense in}\;\hil.
\end{equation}
If $B$ has closed range, then the range-type characterization of absolute continuity takes a much simpler form:
\begin{equation}\label{E:clranchar}
    A\ll B\qiff \ran A\subseteq \ran B, \qquad \mbox{provided that } \ran B = \overline{\ran B}.
\end{equation}

In this paper we are going to investigate singularity and absolute continuity preserving bijections. We say that a bijective map $\ffi:\bph\to\bph$ preserves absolute continuity in both directions if
\begin{equation*}
    A\ll B\qiff\ffi(A)\ll\ffi(B)\qquad\mbox{for all}\;A,B\in\bph.
\end{equation*}
Similarly, we say that a bijection $\ffi:\bph\to\bph$ preserves singularity in both directions if
\begin{equation*}
    A\perp B\qiff\ffi(A)\perp\ffi(B)\qquad\mbox{for all}\;A,B\in\bph.
\end{equation*}

To formulate our results, we need some further notation. With calligraphic letters we always denote linear (not necessarily closed) subspaces of $\hil$ and we use the symbol $\lath$ for the set of all subspaces. A special subset of $\lath$ formed by operator ranges is denoted by 
\begin{equation*}
    \latoph\coloneqq \lers{\cM\subseteq\hil:\exists\; S \in\bh,~\ran S=\cM}=\{\ran A^{1/2}: A\in\bph\},
\end{equation*}
where the second identity is due to the range equality
\begin{equation}\label{E:range}
\ran S=\ran (SS^*)^{1/2}\quad \mbox{for all }S\in\bh.
\end{equation}
It is known that $\latoph$ forms a lattice and that $\latoph\subsetneqq \lath$, for more information see \cite{FW p260corollary2}.

For every positive integer $n$ we set $\latnh$ and $\latcnh$ to be the set of all $n$-dimensional and $n$-codimensional operator ranges, respectively:
\bit
\item[] $\latnh\coloneqq\big\{\cM\in\latoph : \dim\cM=n\big\}=\big\{\cM\in\lath : \dim\cM=n\big\}$,
\item[] $\latcnh\coloneqq\big\{\cM\in\latoph : \codim\cM=n\big\}$.
\eit
Observe also that $\latcnh$ consists of all $n$ codimensional \emph{closed} subspaces of $\hil$. We use the symbol $\mathbf{B}^n_+(\hil)$ to denote the set of all bounded positive operators with $n$ dimensional range.
We also introduce the following subset of $\bph$ which is associated with an operator range $\cM\in\latoph$:
\begin{equation*}
    \RM\coloneqq \lers{C\in\bph : \rancf=\cM}.
\end{equation*}
Note that $\RM$ is never empty according to \eqref{E:range}. 

\section{Main Theorem}

In this section we state and prove our main result. We give a complete description of bijections that preserve absolute continuity in both directions, and of those that preserve singularity in both directions. It turns out that these maps have the same structure.

\begin{theorem}\label{T:maintheorem}
Let $\hil$ be an infinite dimensional complex Hilbert space and assume that $\phi:\bph\to\bph$ is a bijective map. Then the following four statements are equivalent:
\begin{enumerate}[label=\textup{(\roman*)}]
    \item $\phi$ preserves absolutely continuity in both directions,
    \item $\phi$ preserves singularity in both directions,
    \item there exists a bounded, invertible, linear- or conjugate linear operator $T:\hil\to\hil$ such that

\begin{equation}\label{E:property(iii)}
    \ran \phi(A)^{1/2} = \ran T A^{1/2}\qquad \mbox{for all }A\in\bph,
\end{equation}
    \item there exists a bounded, invertible, linear- or conjugate linear operator $T:\hil\to\hil$ and a family $\{Z_A : A\in\bph\}$ of invertible positive operators such that 
    \begin{equation}\label{E:SAS}
        \phi(A)=(TAT^*)^{1/2}Z_A(TAT^*)^{1/2}\qquad\mbox{for all }
         A\in\bph.
    \end{equation}
\end{enumerate}
\end{theorem}

\begin{proof}
(i)$\Longrightarrow$(ii): Notice that (i) implies $\phi(0) = 0$, since $0$ is the only element in $\bph$ which is $B$-absolutely continuous for all positive operator $B$. Moreover, it is easy to see that we have $A\in\mathbf{B}^1_+(\hil)$ if and only if $$\{C\in\bph : C\ll A, A\nll C\}=\{0\},$$ hence $\phi(\mathbf{B}^1_+(\hil)) = \mathbf{B}^1_+(\hil)$. Assume that $\phi$ satisfies (i) but not (ii), hence there exist $A,B\in\bph$ such that $A\perp B$ but $\phi(A)\not\perp\phi(B)$. In particular, this means that there exists a non-zero vector $f\in \hil$ such that $f\otimes f\leq \phi(A)$ and $f\otimes f\leq \phi(B)$, and hence
$$f\otimes f\ll \phi(A)\qquad\mbox{and}\qquad f\otimes f\ll \phi(B).$$
Since $f\otimes f=\phi(e\otimes e)$ holds with some non-zero vector $e\in \hil$, we obtain
\begin{equation*}
    e\otimes e\ll A\qquad\mbox{and}\qquad e\otimes e\ll B.
\end{equation*}
But this implies $e\in\ran A^{1/2}\cap \ran B^{1/2}$, hence $A\not\perp B$, which is a contradiction.  

(ii)$\Longrightarrow$(iii): The first step is to reformulate the singularity preserving property in terms of operator ranges. For any positive operator $A$ we define the set $$A^{\perp}\coloneqq \{C\in\bph:C\perp A\}.$$
From \eqref{d:singrng} it follows easily that 
\begin{equation*}
    A^\perp=B^\perp\qiff \ran A^{1/2}=\ran B^{1/2} \qquad\mbox{for all}\;A,B\in\bph.
\end{equation*}
Consequently, $\phi$ satisfies
\begin{equation}\label{E:rangepreserve}
\ran A^{1/2}=\ran B^{1/2}\qiff\ran \fa^{1/2}=\ran \fb^{1/2}
\end{equation}
for all $A,B\in\bph$.
We introduce the following map:
\begin{equation*}
\Phi:\latoph\to\latoph, \quad \Phi(\ranaf)\coloneqq\ranfaf,    
\end{equation*}
which is obviously well-defined and bijective. 
From \eqref{d:singrng} and \eqref{E:rangepreserve} it is immediate that $\Phi$ preserves ``zero intersection'' in both directions, i.e., 
$$\cM\cap\cN=\seto\qiff\Phi(\cM)\cap\Phi(\cN)=\seto.$$

Next, our task is to understand $\Phi$. We easily see that 
\begin{equation*}
\cM\subseteq\cN \;\iff\; \lers{\cK\in\latoph:\cK\cap\cN=\seto} \subseteq \lers{\cK\in\latoph:\cK\cap\cM=\seto},
\end{equation*}
and thus $\Phi$ preserves inclusion in both directions:
\begin{equation}\label{E:msubn}
\cM\subseteq\cN\qiff\Phi(\cM)\subseteq\Phi(\cN).
\end{equation}
In particular this implies that $\Phi(\{0\}) = \{0\}$ and $\Phi(\hil) = \hil$.
Notice that we have 
$$\dim\cM=1\qiff\lers{\cN\in\latoph :\cN\subseteq\cM}=\lers{\seto,\cM},$$
hence the restriction $\Phi|_{\lateh}$ is a bijection of $\lateh$ onto itself.
Similarly, we have
$$\dim\cM=2\qiff\big\{\cN:\cN\subsetneqq\cM\big\}\subseteq\lateh\cup\big\{\seto\big\},$$
therefore $\Phi|_{\latkh}\colon\latkh\to\latkh$ is also a bijection. 
Combining these observations, we conclude that $\Phi|_{\lateh}$ is a projectivity, that is, $\Phi$ maps any three coplanar elements to coplanar elements. 
Therefore the fundamental theorem of projective geometry (see  e.g. \cite{Faure2}) can be applied: there exists a semilinear bijection $T:\hil\to\hil$ such that
\begin{equation*}\label{f:semilin}
\Phi(\cM)=T(\cM),\qquad\mbox{for all }\cM\in\lateh.
\end{equation*}

Now, we examine how $\Phi$ acts on a general $\cM\in\latoph\setminus\seto$. By the above properties, for all $\cN\in\lateh$ and $\cM\in\latoph$ we have
$$\cN\subseteq\cM\qiff T(\cN)\subseteq\Phi(\cM)$$
and
$$\cM\cap\cN=\seto\qiff T(\cN)\cap\Phi(\cM)=\seto.$$ Therefore, for all $\cM\in\latoph\setminus\seto$ we have
\begin{align*}
    \Phi(\cM)=\bigcup_{\substack{T(\cN)\subseteq\Phi(\cM),\\ T(\cN)\in\lateh}}T(\cN)=\bigcup_{\substack{\cN\subseteq\cM,\\ \cN\in\lateh}}T(\cN)=T\left(\bigcup_{\substack{\cN\subseteq\cM,\\ \cN\in\lateh}}\cN\right)=T(\cM),
\end{align*}
hence, by the definition of $\Phi$ and $T$ we obtain that
$$\ffi[\RM]=\RTM \qquad \mbox{for all }\cM\in\latoph.$$

All that remains is to prove that the semilinear map $T$ is either linear- or conjugate linear, and bounded.  It is immediate that $T$ and $T^{-1}$ map one-codimensional linear manifolds into one-codimensional ones. Furthermore, a finite codimensional subspace of $\hil$ is an operator range if and only if it is closed, so we infer that  $T$ maps $\latcegyh$ onto $\latcegyh$. Since  $\hil$ is infinite dimensional, we can use \cite[Lemma 2 and its Corollary]{KM}  to conclude that $T$ is either linear- or conjugate linear. Finally, to  show that $T$ is bounded it suffices to prove that $y^*\circ T$ is bounded for every bounded linear functional $y^*\in \hil^*$. Suppose first that $T$ is linear. Since $T$ maps $\latcegyh$ onto $\latcegyh$, there is $x^*\in \hil^*$ such that $\ker y^*=T(\ker x^*)$. Consequently, $\ker x^*= \ker (y^*\circ T)$, which implies $y^*\circ T=\lambda x^*$ for some $\lambda$, and hence $y^*\circ T$ is bounded. A very similar approach applies when $T$ is conjugate linear.

(iii)$\Longrightarrow$(iv): First, assume that $T$ is linear. Then by \eqref{E:range} we obtain
\begin{equation}\label{E:ranTAT}
    \ran \phi(A)^{1/2}=\ran TA^{1/2}=\ran (TAT^*)^{1/2}\qquad \mbox{for all $A\in\bph$}.
\end{equation}
Hence by \cite[Corollary 1 on p.259]{FW p260corollary2} we have $\phi(A)^{1/2}=(TAT^*)^{1/2}X_A$ with some invertible operator $X_A\in\bh$. Therefore \eqref{E:SAS} clearly holds with $Z_A = X^{}_AX_A^*$.

Assume now that $T$ is conjugate linear. Consider an arbitrary antiunitary operator $U:\hil\to\hil$. Then
$$ \ran \phi(A)^{1/2} = \ran  T A^{1/2} = \ran TA^{1/2}U = \ran (TAT^*)^{1/2}$$
for all $A\in\bph$, where in the last step we used \eqref{E:range} for the linear bounded operator $T A^{1/2}U$. From here we finish the proof as in the linear case. 

(iv)$\Longrightarrow$(i): By \eqref{E:range} we have
\begin{equation*}
    \ran \phi(A)^{1/2}=\ran (TAT^*)^{1/2}Z_A^{1/2}=\ran (TAT^*)^{1/2}=\ran TA^{1/2}=\ran TA^{1/2}T^*
\end{equation*}
for all $A\in\bph$. Thus by \cite[Corollary 1 on p.259]{FW p260corollary2}, there exists an invertible operator $Y_A\in\bh$ such that
\begin{equation*}
    \phi(A)^{1/2}=TA^{1/2}T^*Y_A.
\end{equation*}
If we introduce the notation $\mathscr{D}_{A,B}\coloneqq \{x\in\hil : A^{1/2}x\in\ran B^{1/2}\}$
for every pair $A,B\in\bph$, then an immediate calculation shows that 
\begin{equation*}
    \mathscr{D}_{\phi(A),\phi(B)}= (T^*Y_A)^{-1}(\mathscr{D}_{A,B})
\end{equation*}
from which it follows that $\mathscr{D}_{A,B}$ is dense if and only if $\mathscr{D}_{\phi(A),\phi(B)}$ is dense. By \eqref{d:acrng} this implies (i). 
\end{proof}

\section{The finite dimensional case}
If $\dim\hil <\infty$, then $\lath=\latoph$, every operator has closed range, and $\ran A=\ran A^{1/2}$ holds for all $A\in\bph$. Therefore the notions of absolute continuity and singularity simplify considerably.
In particular, the characterization \eqref{E:clranchar} of absolute continuity is valid for every pair $A,B$ of positive operators. Similarly, the range characterization of  singularity reduces to   
\begin{equation*}
    A\perp B\qiff \ran A\cap \ran B=\{0\}.
\end{equation*}
Furthermore, we have $\RM=\{C\in\bph : \ran C=\cM\}$ for all $\cM\in\lath$. 
Therefore the finite dimensional version of Theorem A can be proved much more easily using the fundamental theorem of projective geometry provided that $\dim H > 2$. However, we point out that the result we get is slightly different, as $T$ is not necessarily linear- or conjugate linear anymore. We omit the proof.

\begin{theorem}\label{T:finitedim}
Let $\hil$ be a complex Hilbert space such that $3\leq \dim \hil< +\infty$ and let $\phi:\bph\to\bph$ be a bijective map. Then the following three statements are equivalent:
\begin{enumerate}[label=\textup{(\roman*)}]
    \item $\phi$ preserves absolutely continuity in both directions,
    \item $\phi$ preserves singularity in both directions,
    \item there is a semilinear bijection $T:\hil\to\hil$ such that
    $$ \ran \phi(A)=\ran TA \quad \mbox{for all } A\in\bph.$$
\end{enumerate}
\end{theorem}

Finally, in case when $\dim \hil = 2$, the fundamental theorem of projective geometry cannot be applied. However, one can prove easily that points (i) and (ii) are both equivalent with the following condition:
\begin{itemize}
    \item[(iii')] $\phi(0) = 0$, $\phi$ maps the set of all invertible positive operators bijectively onto itself, and there is a bijection $\Psi\colon \lateh\to\lateh$ such that 
    $$\ran \phi(A) = \Psi(\ran A) \quad \mbox{for all } A\in\mathbf{B}^1_+(\hil).$$
\end{itemize}
\section*{Acknowledgement}
We would like to thank the referee for his/her helpful comments on the paper.

\end{document}